\documentclass{birkjour}

\usepackage[utf8]{inputenc}
\usepackage[T1]{fontenc}
\usepackage{amsmath, amsfonts, amssymb, amsthm}
\usepackage[none]{hyphenat}
\usepackage{graphicx}
\usepackage{float}
\usepackage{times}
\usepackage{esint}
\usepackage{lipsum}
\usepackage{enumitem}
\usepackage{amsmath, amsthm, amscd, amsfonts, amssymb, graphicx, color}
\usepackage[bookmarksnumbered, colorlinks, plainpages]{hyperref}
\hypersetup{colorlinks=true,linkcolor=red, anchorcolor=green, citecolor=cyan, urlcolor=red, filecolor=magenta, pdftoolbar=true}
\usepackage{orcidlink}
\usepackage[scr]{rsfso}
\usepackage{hyperref}
\newtheorem{theorem}{Theorem}[section]
\newtheorem{lemma}[theorem]{Lemma}
\newtheorem{proposition}[theorem]{Proposition}

\theoremstyle{definition}
\newtheorem{definition}[theorem]{Definition}

\newtheorem{remark}[theorem]{Remark}
\numberwithin{equation}{section}

\numberwithin{equation}{section}

\begin{document}

%-------------------------------------------------------------------------
% editorial commands: to be inserted by the editorial office
%
%\firstpage{1} \volume{228} \Copyrightyear{2004} \DOI{003-0001}
%\seriesextraline{This is the Concrete Title of this Book\br H.E. R and S.T.C. W, Eds.}
%\firstpage{1}

\title{ On the singular elliptic problem involving the variable order fractional Musielak $g_{x,y}$-Laplacian}

%----------Author 1
\author{Azeddine BAALAL}
\address{{ Department of Mathematics and Computer Science},
	{A\"in Chock Faculty,  Hassan II University}, 
	{ B.P. 5366 Maarif, Casablanca},
	{Morocco}}
\email{abaalal@gmail.com}

%----------Author 2
\author{Mohamed BERGHOUT}
 	\address{{Laboratory of Partial Differential Equations, Algebra and Spectral Geometry, Higher School of Education and Training, Ibn Tofail University}, {P.O. Box 242-Kenitra 14000, Kenitra}, {Morocco}}
\email{Mohamed.berghout@uit.ac.ma; moh.berghout@gmail.com}
 
 %----------Author 3
 \author{El-Houcine OUALI* \orcidlink{0009-0007-9106-6441}}
\address{{ Department of Mathematics and Computer Science},
	{A\"in Chock Faculty,  Hassan II University}, 
	{ B.P. 5366 Maarif, Casablanca},
	{Morocco}}
\email{oualihoucine4@gmail.com}


\date{}

\begin{abstract}
	  In this article, we investigate the existence of weak solutions for the following singular nonlinear elliptic problem :
	  $$\left\{\begin{aligned}(-\Delta)^{s(.,.)}_{G_{x,y}} u(x) +a_{x,y}(|u(x)|)u(x) & = \frac{h(x)}{u^{m(x)}} & & \text { in } \Omega\\ u & >0 & & \text { in } \Omega \\ u & =0 & & \text { in } \mathbb{R}^N \backslash \Omega,\end{aligned}\right.$$
where $N \geq 2, \Omega \subset \mathbb{R}^N$ is a bounded domain with Lipschitz boundary $\partial \Omega$, $m:\overline{\Omega}\rightarrow (0,1)$ is a continuous function, $p:\overline{\Omega}\times\overline{\Omega}\rightarrow (0,1)$ is is a bounded, continuous and symmetric function and $h: \overline{\Omega} \rightarrow \mathbb{R}$ is a positive continuous function belonging to $ L^{\frac{r(x)}{r(x)+m(x)-1}}(\Omega)$ where $r:\overline{\Omega}\rightarrow (1,+\infty)$ is a continuous function such that $1<r(x)<p_s^*(x):=\frac{N p(x,x)}{N-s(x,x) p(x,x)}$. The main tool is variational approach, however,
 various auxiliary tools from the theory of nonlinear functional analysis, convex analysis
 and critical point theory are also applied.
\end{abstract}

%%% ----------------------------------------------------------------------
\maketitle
%%% ----------------------------------------------------------------------

\section{Introduction and main results}
 In this work, we investigate the existence of weak solutions to the following singular nonlinear elliptic problem :
 \begin{equation}\label{Prob1}
 \left\{\begin{aligned}(-\Delta)^{s(.,.)}_{G_{x,y}} u(x) +a_{x,y}(|u(x)|)u(x) & = \frac{h(x)}{u^{m(x)}} & & \text { in } \Omega \\  u & >0 & & \text { in } \Omega \\u & =0 & & \text { in } \mathbb{R}^N \backslash \Omega,\end{aligned}\right.
  \end{equation}
 where $N \geq 2, \Omega \subset \mathbb{R}^N$ is a bounded domain with Lipschitz boundary $\partial \Omega$, $h: \overline{\Omega} \rightarrow \mathbb{R}$ is a positive continuous function belonging to $ L^{\frac{r(x)}{r(x)+m(x)-1}}(\Omega)$ where $r:\overline{\Omega}\rightarrow (1,+\infty)$ is a continuous function satisfying $1<r(x)<p_s^*(x):=\frac{N p(x,x)}{N-s(x,x) p(x,x)}$, $m:\overline{\Omega}\rightarrow (0,1)$ is a continuous function, $s:\mathbb{R}^N\times\mathbb{R}^N\rightarrow (0,1)$ is a continuous function such that 
\begin{equation}\label{S1}
s(x, y)=s(y, x) \quad \forall (x, y) \in \mathbb{R}^N\times\mathbb{R}^N,
\end{equation}
\begin{equation}\label{S2}
0<s^{-}=\inf _{\mathbb{R}^N\times\mathbb{R}^N} s(x, y) \leqslant s^{+}=\sup _{\mathbb{R}^N\times\mathbb{R}^N} s(x, y)<1
\end{equation}
and \((-\Delta)^{s(.,.)}_{G_{x,y}}\) is the variable order fractional $G_{x,y}$-Laplacian operator given by

\[
(-\Delta)^{s(.,.)}_{G_{x,y}} u(x) := 2 \lim_{\epsilon \to 0^+} \int_{\mathbb{R}^N \setminus B_\epsilon (x)} a_{x,y} \left( \frac{|u(x) - u(y)|}{|x - y|^{s(x,y)}} \right) \frac{u(x) - u(y)}{|x - y|^{s(x,y)}} \frac{dy}{|x - y|^{N+s(x,y)}}.
\]
where $G: \overline{\Omega} \times \overline{\Omega} \times \mathbb{R} \rightarrow \mathbb{R}$ is a Carathéodory function defined by
\begin{equation}\label{eq1}
G_{x,y}(t):=G(x, y, t)=\int_0^{|t|} g(x, y, \tau) d \tau,
\end{equation}
and
$$g(x, y, t):= \begin{cases}a(x, y, t) t & \text { if } t \neq 0 \\ 0 & \text { if } t=0,\end{cases}$$
with $a: \overline{\Omega} \times \overline{\Omega} \times(0, +\infty) \rightarrow[0, +\infty)$ is a function satisfying :
\begin{itemize}
\item[$({g_{1}})$] $\displaystyle\lim _{t \rightarrow 0} a(x, y, t)t=0$ and $\displaystyle\lim _{t \rightarrow +\infty} a(x, y, t)t=+\infty$ for all $(x, y) \in \overline{\Omega} \times \overline{\Omega}$;
\item[$(g_{2})$] $t\mapsto a_{x,y}(t):=a(x, y, t)$ is continuous on $(0, +\infty)$ for all $(x, y) \in \overline{\Omega} \times \overline{\Omega}$;
\item[$(g_{3})$] $t\mapsto a_{x,y}(t)t$ is increasing on $(0, +\infty)$ for all $(x, y) \in \overline{\Omega} \times \overline{\Omega}$.
\item[\noindent\textbf{$(g_{4})$}\label{cond:G1}] There exist positive constants $g^{+}$and $g^{-}$ such that
\begin{equation}\label{eq1.5}
1<g^{-} \leqslant \frac{ a_{x, y}(t)t^{2}}{G_{x, y}(t)} \leqslant g^{+}<+\infty \quad \textsl{ for all } (x, y) \in \overline{\Omega} \times \overline{\Omega}, \quad  \textsl{ and all  } t > 0 .
\end{equation}
\end{itemize}

Now let us consider the function $\widehat{G}_x: \overline{\Omega} \times \mathbb{R}\rightarrow \mathbb{R} $ given by
\begin{equation}\label{eq2}
\widehat{G}_x(t):=\widehat{G}(x, t):=G(x,x,t)=\int_0^{|t|} \widehat{g}(x,\tau) \mathrm{d} \tau .
\end{equation}
where $\widehat{g}(x,t):=\widehat{a}(x,t)t=a(x, x,t)t$ for all $(x, t) \in \overline{\Omega} \times (0,+\infty)$.
 The assumption (\hyperref[cond:G1]{$g_{4}$}) implies that
\begin{equation}\label{eq1.6}
1<g^{-} \leqslant \frac{ \widehat{g}(x,t)t}{\widehat{G}(x,t)} \leqslant g^{+}<+\infty, \quad \textsl{ for all } x \in \overline{\Omega},\quad \textsl{ and all } t > 0 .
\end{equation}\\
As a consequence of the previous assumptions, one may obtain the following properties:
\begin{itemize}
\item[(i)] $t\rightarrow G_{x,y}(t)$ is  even, continuous,  strictly increasing and convex on $\mathbb{R}$;
\item[(ii)] $\displaystyle\lim _{t \rightarrow 0} \frac{G_{x,y}(t)}{t}=0$;
\item[(iii)]
$
\displaystyle\lim _{t \rightarrow +\infty} \frac{G_{x,y} (t)}{t}=+\infty$;
\item[(iv)] $G_{x,y}(t)>0$ for all $t>0$.
\end{itemize}
See for instance the book of Kufner–John-Fu\v{c}\'{\i}k \cite[Lemma 3.2.2]{kajofs1977}.
\begin{definition}\label{Def1.1}
Let $\Omega$ be an open subset of $\mathbb{R}^{N}$. A function $G: \overline{\Omega} \times \overline{\Omega} \times \mathbb{R} \rightarrow \mathbb{R}$ is said to be a generalized N-function if it fulfills the properties (i)$-$(iv) above for a.e. $(x, y) \in \overline{\Omega} \times \overline{\Omega}$, and for each $t\in \mathbb{R}$, $G_{x,y}(t)$ is measurable in $(x, y)$.
\end{definition}

In recent years, equations involving non-local operators have attracted increasing attention due to their wide range of applications in areas such as mechanics, phase transitions, population dynamics, image processing, and game theory; see, for instance, \cite{caffarelli2012non,caffarelli2007extension,caffarelli2011uniform,metzler2004restaurant} and the references therein. An extensive body of literature is devoted to such problems with singularities, particularly from the perspective of theoretical analysis. For example, in \cite{bhobrm2023anewclass}, the authors studied the existence of weak solutions to the following singular elliptic problem
  $$\left\{\begin{aligned}(-\Delta)^{s(.,.)}_{G} u(x)& = h(x)f^{'}(x,|u|)\frac{u}{|u|} & & \text { in } \Omega \\ u & =0 & & \text { in } \mathbb{R}^N \backslash \Omega,\end{aligned}\right.$$
 where $\Omega$ is a bounded domain with Lipschitz boundary, $f^{'}$ is a generalized singular term and $h$ is a positive function. In \cite{bal2024singular}, the authors show the existence of solutions for the following
 problem  $$\left\{\begin{aligned}(-\Delta_{g})^{s}u(x)& = \frac{f(x)}{u^{m(x)}} & & \text { in } \Omega \\ u & >0 & & \text { in } \Omega \\ u & =0 & & \text { in } \mathbb{R}^N \backslash \Omega,\end{aligned}\right.$$
 where $\Omega$ is a smooth bounded domain in $\mathbb{R}^N$, 
$m \in C^{1}(\overline{\Omega})$, and $(-\Delta_{g})^{s}$ 
is the fractional $g$-Laplacian with $g$ the antiderivative of a Young 
function and $f$ in suitable Orlicz space. In \cite{chammem2023nehari}, the authors show the existence of solutions for the following problem
 \[
\begin{cases}
(-\Delta)^{s}_{p(x,.)} u + \mu |u|^{q(x)-2}u = \lambda g(x)u^{-m(x)} +\lambda f(x,u), & \text{ in } \Omega, \\
u = 0, & \text{ on } \partial \Omega, 
\end{cases}
\]
where $\Omega$ is a bounded domain in $\mathbb{R}^N$, $(-\Delta)^{s}_{p(x,.)}$ 
is the fractional $p(x,.)$-Laplacian operator, $\mu$ and $\lambda$ are positive parameters and $m:\overline{\Omega}\longrightarrow (0,1)$ is a continuous function. Notice that  the
 fractional Musielak-Sobolev spaces with variable order $W^{s(.,.),G_{x, y}}(\Omega)$ (see \cite{ouali2025some,srtM2024eignv}) represent a generalization of the fractional Orlicz Sobolev spaces $W^{s,G}(\Omega)$ (see \cite{bahrouni2020embedding,bot2020basic,BjSa2019}) and of the fractional Sobolev spaces with variable exponents $W^{s,q(.),p(.,.)}(\Omega)$ (see \cite{baalal2018traces,babm2018density,barv2018onanew,dlrj2017traces,kjv2017fract,km2023bourgan}). Consequently, many properties originally established for $W^{s,G}(\Omega)$ and $W^{s,q(.),p(.,.)}(\Omega)$ have been extended to also cover $W^{s(.,.),G_{x, y}}(\Omega)$. In addition, the nonlocal integro-differential operator $(-\Delta)^{s(.,.)}_{G_{x,y}}$ associated with the fractional Musielak-Sobolev spaces with variable order, can be viewed as a generalization of the fractional $p(x,·)$-Laplacian operator $(\Delta)^{s}_{p(x,.)}$ when $G_{x,y}(t)=|t|^{p(x,y)}$ and $s(x,y)=s=$constant and of the fractional $g$-Laplacian operator $(-\Delta_{g})^{s}$ when $G_{x,y}(t)=G$ i.e. $G$	 is independent of variables $x$, $y$, and $s(x,y)=s=$constant.
 
 Inspired by the studies mentioned above, this paper aims to extend this class of problems to the framework of fractional Musielak–Sobolev spaces with variable order. In particular, we consider the nonlocal singular elliptic problem \eqref{Prob1} and establish the existence of a positive ground state solution with a negative energy level. To the best of our knowledge, this is the first contribution addressing nonlocal singular problems within this class of fractional Musielak–Sobolev spaces with variable order.
\par Let $G_{x,y}$ be a generalized N-function, $\Omega$ be an open subset of $\mathbb{R}^{N}$ and $s(.,.): \mathbb{R}^N\times\mathbb{R}^N \rightarrow(0,1)$ a continuous function satisfying \eqref{S1} and \eqref{S2}.
In correspondence to $\widehat{G}_{x}=G_{x,x}$, the Musielak-Orlicz space $L^{\widehat{G}_x}(\Omega)$ is defined as follows  
$$
L^{\widehat{G}_x}(\Omega)=\left\{u: \Omega \longrightarrow \mathbb{R} \text { mesurable : } J_{\widehat{G}_x}(\lambda u)<+\infty \text { for some } \lambda>0\right\},
$$
where $J_{\widehat{G}_x}( u):=\displaystyle\int_{\Omega} \widehat{G}_x(|u(x)|) \mathrm{d} x,$
and the $s(.,.)$-fractional Musielak-Sobolev space $W^{s(.,.),G_{x, y}}(\Omega)$ is defined by
$$
\begin{aligned}
&W^{s(.,.),G_{x, y}}(\Omega):=\left\{u \in L^{\widehat{G}_x}(\Omega): J_{s(.,.),G_{x, y}}( \lambda u)<+\infty \quad\textnormal{ for some } \lambda>0 \right\},
\end{aligned}
$$
where  $$J_{s(.,.),G_{x, y}}(u):= \int_{\Omega} \int_{\Omega} G_{x, y}\left(D_{s(.,.)}u\right)d\mu,$$\\
with $D_{s(.,.)}u:=D_{s(.,.)}u(x,y):=\displaystyle\frac{u(x)-u(y)}{|x-y|^{s(x,y)}}$ and $d\mu:=\displaystyle\frac{\mathrm{d} x \mathrm{~d} y}{|x-y|^{N}}.$

\begin{remark}
\begin{itemize}
\item[$(a)$] For the case: $s(x, y)=s$ i.e. $s$ is constant, $W^{s,G_{x,y}}(\Omega)$ is the fractional Musielak-Sobolev space, (see \cite{abss2022class, abss2023emb}), defined by	
$$
\begin{aligned}
&W^{s,G_{x,y}}(\Omega):=\left\{u \in L^{\widehat{G}_{x}}(\Omega): J_{s,G_{x,y}}( \lambda u)<+\infty \quad\textnormal{ for some } \lambda>0 \right\},
\end{aligned}
$$                                                                                                                                                                 where 	$$J_{s,G_{x,y}}(u):= \int_{\Omega} \int_{\Omega} G_{x,y}\left(D_{s}u(x,y)\right)d\mu.$$	
\item[$(b)$] For the case: $G_{x, y}(t)=G(t)$ i.e. $G$	 is independent of variables $x$, $y$, we say that $L^{G}(\Omega)$ and $W^{s(.,.),G}(\Omega)$ are Orlicz spaces and $s(.,.)$-fractional Orlicz-Sobolev spaces respectively (see \cite{bhobrm2023anewclass}) such that	
$$
\begin{aligned}
&W^{s(.,.),G}(\Omega):=\left\{u \in L^{G}(\Omega): J_{s(.,.),G}( \lambda u)<+\infty \quad\textnormal{ for some } \lambda>0 \right\},
\end{aligned}
$$                                                                                                                                                                 where 	$$J_{s(.,.),G}(u):= \int_{\Omega} \int_{\Omega} G\left(D_{s(.,.)}u(x,y)\right)d\mu.$$												\item[$(c)$] For the case: $G_{x, y}(t)=|t|^{p(x, y)}$ for all $(x, y) \in \overline{\Omega} \times \overline{\Omega}$, where $p: \overline{\Omega} \times \overline{\Omega} \longrightarrow(1,+\infty)$ is a continuous function such that
\begin{equation}\label{P1}
1<p^{-} \leqslant p(x, y) \leqslant p^{+}<+\infty
\end{equation}
and
\begin{equation}\label{P2}
p \text { is symmetric, that is, } p(x, y)=p(y, x)\quad \forall(x, y) \in \overline{\Omega} \times \overline{\Omega}, \textnormal{ with } s^{+}p^{+}<N,
\end{equation}
if denoted by $\overline{p}(x)=p(x, x)$ for all $x \in \overline{\Omega}$, then we replace $L^{\widehat{G}_x}$ by $L^{\overline{p}(x)}$ and $W^{s(.,.),G_{x, y}}$ by $W^{s(.,.), p(x, y)}$ and we refer them as variable exponent Lebesgue spaces and $s(.,.)$-fractional Sobolev spaces with variable exponent respectively (see \cite{zuo2021critical}) defined by
$$
L^{\overline{p}(x)}(\Omega)=\left\{u: \Omega \longrightarrow \mathbb{R} \text { measurable }: \int_{\Omega}|u(x)|^{\overline{p}(x)} \mathrm{d} x<+\infty\right\}
$$
and
$$
W^{s(.,.), p(x, y)}(\Omega)=\left\{u \in L^{\overline{p}(x)}(\Omega): \rho_{s(.,.),p(.,.)}(\lambda u)<+\infty \text { for some } \lambda>0\right\}
$$
where	$$\displaystyle\rho_{s(.,.),p(.,.)}( u):=\int_{\Omega}\int_{\Omega} \frac{|u(x)-u(y)|^{p(x, y)}}{|x-y|^{s(x,y) p(x, y)+N}} \mathrm{~d} x \mathrm{~d} y.$$																	
\end{itemize}										
\end{remark}

In light of assumption (\hyperref[cond:G1]{$g_{4}$}), the functions $G_{x, y}$ and $\widehat{G}_x$ satisfy the $\Delta_2$-condition (see \cite[Proposition 2.3]{mmrv2008neu}), written $G_{x, y} \in \Delta_2$ and $\widehat{G}_x \in \Delta_2$, that is there exists a positive constant $K$ such that
\begin{equation}
G_{x, y}(2 t) \leqslant K G_{x, y}(t) \quad \textnormal{ for all  }\quad (x, y) \in \overline{\Omega} \times \overline{\Omega} \quad \textnormal{ and  }\quad t > 0 \text {, }
\end{equation}
and
\begin{equation}\label{1.6eq2}
\widehat{G}_x(2 t) \leqslant K \widehat{G}_x(t) \quad \textnormal{ for all  }\quad x \in \overline{\Omega} \quad \textnormal{ and  }\quad t > 0.
\end{equation}

\begin{definition}\label{Def1.B}
We say that a generalized N-function $G_{x,y}$ satisfies the fractional boundedness condition, writen $G_{x,y}\in \mathcal{B}_{f} $, if there exist $C_{1}, C_{2} > 0$ such that
\begin{equation}
 C_{1}\leqslant G_{x,y}(1)\leqslant C_{2} \quad \forall (x,y)\in \overline{\Omega}\times \overline{\Omega}.
 \end{equation}
\end{definition}
\begin{definition}\label{DefinA} Let $\widehat{A}_x$ and $\widehat{B}_x$ be two generalized N-functions. We say that $\widehat{A}_x$ essentially grows more slowly than $\widehat{B}_x$ near infinity, and we write $\widehat{A}_x \prec \prec \widehat{B}_x$, if for all $k>0$, we have

$$
\lim _{t \rightarrow+\infty} \frac{\widehat{A}_x(k t)}{\widehat{B}_x(t)}=0, \quad \text { uniformly in } x \in \overline{\Omega}.
$$
\end{definition}
																								
Now, we give some examples of generalized N-functions satisfying the previous assumptions:
\begin{itemize}
\item[(1)] Let $G_{x, y}(t)=\frac{1}{p(x,y)} |t|^{p(x, y)}$, for all $(x, y) \in \overline{\Omega} \times \overline{\Omega}$ and all $t \in \mathbb{R}$, where $p: \overline{\Omega} \times \overline{\Omega} \longrightarrow(1,+\infty)$ is a continuous function satisfying
$$
1<p^{-} \leq p(x, y) \leq p^{+}<+\infty, \text {for all }(x, y) \in \overline{\Omega} \times \overline{\Omega}
$$
and
$$p\left( (x, y)-(z,z)\right)=p(x, y),\quad \text {for all }\quad(x, y),(z,z) \in \overline{\Omega} \times \overline{\Omega}.  $$
In this case the function $G_{x,y}$ satisfies the assumptions $(\hyperref[cond:G1]{g_{1}})-(\hyperref[cond:G1]{g_{4}})$ and $G_{x,y}\in \mathcal{B}_{f}$.
\item[(2)] Let $a$ be a real-valued function defined on $[0, +\infty)$ and having the following properties:
\begin{itemize}
\item[(a)] $a(0)=0, a(t)>0$ if $t>0, \lim _{t \rightarrow +\infty} a(t)=+\infty$;\\
\item[(b)] $a$ is nondecreasing, that is, $s>t$ implies $a(s) \geq a(t)$;\\
\item[(c)] $a$ is right continuous, that is, if $t \geq 0$, then $\lim _{s \rightarrow t+} a(s)=a(t)$.\\		
\end{itemize}	
Let $G_{x, y}(t)=A(t)$, for all $(x, y) \in \overline{\Omega} \times \overline{\Omega}$ and all $t \geqslant 0$, where $\displaystyle A(t):=\int_0^t a(\tau) d \tau$ is an N-function (for definition see \cite{Adams1975}) satisfying the following condition
$$1<a^{-} \leq \frac{a(t) t}{A(t)} \leq a^{+}<a^{-}_{*}:=\frac{Na^{-}}{N-sa^{-}}<+\infty \quad \textnormal{ for all }\quad t>0$$
and $$\displaystyle \int_0^1 \frac{A^{-1}(\tau)}{\tau^{\frac{N+s}{N}}} d \tau<+\infty \quad \textnormal{ and} \quad \displaystyle\int_1^{+\infty} \frac{A^{-1}(\tau)}{\tau^{\frac{N+s}{N}}} d \tau=+\infty.$$
It is clear that the generalized N-function $G_{x, y}$ satisfies the assumptions $(\hyperref[cond:G1]{g_{1}})-(\hyperref[cond:G1]{g_{4}})$ and $G_{x,y}\in \mathcal{B}_{f}$.
\end{itemize}
Notice that these particular cases have applications in several fields of physics and mathematics, for example, filtration of fluids in porous media, restricted heating, elastoplasticity, image processing, optimal control and financial mathematics, (see for instance \cite{aberqi2022new,dlhphp2011,BjSa2019}).
\begin{definition}
Let $G_{x,y}$ be a generalized N-function. The function $\widetilde{G}: \overline{\Omega} \times \overline{\Omega} \times [0,+\infty) \rightarrow\mathbb{R}$ defined by
\begin{equation}\label{Gconjugate}
\widetilde{G}_{x, y}(t)=\widetilde{G}(x, y, t):=\sup _{s \geq 0}\left(t s-G_{x, y}(s)\right) \quad  \forall  (x, y) \in \overline{\Omega} \times \overline{\Omega}, \quad  \forall t \geqslant 0
\end{equation}
is called the conjugate of $G$ in the sense of Young.
\end{definition}

The assumptions $(\hyperref[cond:G1]{g_{1}})-(\hyperref[cond:G1]{g_{4}})$ ensure that $\widetilde{G}_{x,y}$ is a generalized $N$-function (see \cite{phph2019}) and satisfies the $\Delta_2$-condition (see \cite{fukagai2006positive}). Moreover, we have  $\displaystyle\widetilde{g}^{-} \leq \frac{\widetilde{\widehat{g}}(x, t) t}{\widehat{G}(x, t)} \leq \widetilde{g}^{+}$, for all $x \in \overline{\Omega}$ and all $t>0$ where $\displaystyle\widetilde{g}^{-}=\frac{g^{+}}{g^{+}-1}$ and $\displaystyle\widetilde{g}^{+}=\frac{g^{-}}{g^{-}-1}$.

 In view of \eqref{Gconjugate} we have the following Young's type inequality:
\begin{equation}\label{ineqYong}
\sigma \tau \leq G_{x, y}(\sigma)+\widetilde{G}_{x, y}(\tau), \text { for all }(x, y) \in \overline{\Omega} \times \overline{\Omega} \text { and } \sigma, \tau \geq 0.
\end{equation}

\subsection{Main results}
\begin{definition}
 A function $u$ is called a weak solution to problem \eqref{Prob1} if $u \in W^{s(.,.),G_{x,y}}(\Omega)$ such that $u>0$ in $\Omega$ and

\begin{equation}\label{Eqt1}
\int_{\Omega}\int_{\Omega} a_{x,y}\left(\left|D_{s(.,.)}u\right|\right) D_{s(.,.)}u D_{s(.,.)}v d \mu+\int_{\Omega} \widehat{a}_{x}(|u|) u v d x=\int_{\Omega} \frac{h(x)}{u^{m(x)}} v d x
\end{equation}
for all $v \in W^{s(.,.),G_{x,y}}(\Omega)$.
\end{definition}
Our main result reads as follows.
\begin{theorem}\label{Theopr}

Let $\Omega \subset \mathbb{R}^N$ be a bounded domain with Lipschitz boundary $\partial \Omega$. Suppose that $(\hyperref[cond:G1]{g_{1}})-(\hyperref[cond:G1]{g_{4}})$ and \eqref{ConditionN} hold and $G_{x,y}\in \mathcal{B}_{f}$. Then the problem \eqref{Prob1} has a nontrivial weak solution $u_{0}\in W^{s(.,.),G_{x, y}}(\Omega)$.
\end{theorem}																										
This paper is organized as follows.  In Section \ref{sec2}, we state some fundamental properties of the generalized N-functions, Musielak-Orlicz spaces and fractional Musielak-Sobolev spaces with variable order. In Section \ref{sec3} we give the proof of the main results.

\section{ Some preliminary results}\label{sec2}
In this section we give some definitions and properties for the fractional Musielak-Sobolev spaces with variable order, as well as some preliminary lemmas, which will be used in the sequel.

\subsection{Musielak-Orlicz spaces}
Let $G_{x,y}$ be a generalized N-function. In correspondence to $\widehat{G}_x=G_{x,x}$ and an open subset $\Omega$ of $\mathbb{R}^{N}$, we define the Musielak class 
$$
K^{\widehat{G}_x}(\Omega)=\left\{u: \Omega \longrightarrow \mathbb{R} \text { mesurable : } \int_{\Omega} \widehat{G}_x(|u(x)|) \mathrm{d} x<+\infty\right\} \text {, }
$$
and we recall the Musielak-Orlicz space
$$
L^{\widehat{G}_x}(\Omega)=\left\{u: \Omega \longrightarrow \mathbb{R} \text { mesurable : } J_{\widehat{G}_x}(\lambda u)<+\infty \text { for some } \lambda>0\right\},
$$
where $$J_{\widehat{G}_x}( u):=\displaystyle\int_{\Omega} \widehat{G}_x(|u(x)|) \mathrm{d} x.$$\\
The space $L^{\widehat{G}_x}(\Omega)$ is endowed with the Luxemburg norm
$$
\|u\|_{L^{\widehat{G}_{x}}(\Omega)}=\inf \left\{\lambda>0: J_{\widehat{G}_x}\left(\frac{u}{\lambda} \right)  \leqslant 1\right\} .
$$
 Under the assumptions $(\hyperref[cond:G1]{g_{1}})-(\hyperref[cond:G1]{g_{4}})$, the space $\left(L^{\widehat{G}_{x}}(\Omega),\|u\|_{L^{\widehat{G}_{x}}(\Omega)}\right)$ is a separable and reflexive Banach space (see \cite[Remark B.1]{ayya2020someapp} and \cite[Theorem 3.3.7 and Corollary 3.6.7]{phph2019}). The relation \eqref{1.6eq2} implies that $ L^{\widehat{G}_x}(\Omega)=K^{\widehat{G}_x}(\Omega)$ (see \cite{musju1983Orlicz}).

As a consequence of \eqref{ineqYong}, we have the following Hölder's type inequality:
\begin{lemma}
 Let $\Omega$ be an open subset of $\mathbb{R}^N$. Let $\widehat{G}_x$ be a generalized $N$-function and $\widetilde{\widehat{G}}_x$ its conjugate function, then
$$
\displaystyle\left|\int_{\Omega} u v d x\right| \leq 2\|u\|_{L^{\widehat{G}_x}(\Omega)}\|v\|_{L^{\widetilde{\widehat{G}}_{x}}(\Omega)},
$$
for all $u \in L^{\widehat{G}_x}(\Omega)$ and all $ v \in L^{\widetilde{\widehat{G}}_{x}}(\Omega).$
\end{lemma}
\begin{lemma}[{\textnormal{see} \cite{phph2019}}]\label{Lemm2} Let $\widehat{G}_x$ be a generalized $N$-function. Then, we have $$\|u\|_{L^{\widehat{G}_x}(\Omega)} \leq J_{\widehat{G}_x}(u)+1, \quad$$ for all $u \in L^{\widehat{G}_x}(\Omega)$.
\end{lemma}

\subsection{$s(.,.)$-Fractional Musielak-Sobolev spaces}
Let $G_{x,y}$ be a generalized N-function and $\Omega$ an open subset of $\mathbb{R}^{N}$. We recall the fractional Musielak-Sobolev space with variable order
$$
\begin{aligned}
&W^{s(.,.),G_{x, y}}(\Omega):=\left\{u \in L^{\widehat{G}_x}(\Omega): J_{s(.,.),G_{x, y}}( \lambda u)<+\infty \quad\textnormal{ for some } \lambda>0 \right\},
\end{aligned}
$$
where  $$J_{s(.,.),G_{x, y}}(u):= \int_{\Omega} \int_{\Omega} G_{x, y}\left(D_{s(.,.)}u\right)d\mu,$$
with $D_{s(.,.)}u:=D_{s(.,.)}u(x,y):=\displaystyle\frac{u(x)-u(y)}{|x-y|^{s(x,y)}}$ and $d\mu:=\displaystyle\frac{\mathrm{d} x \mathrm{~d} y}{|x-y|^{N}}.$\\
It is well known that $d\mu$ is a regular Borel measure on the set $\Omega\times\Omega$.

The space $W^{s(.,.),G_{x, y}}(\Omega)$ is endowed with the norm
\begin{equation}\label{Norm}
\|u\|_{W^{s(.,.),G_{x, y}}(\Omega)}:=\|u\|_{L^{\widehat{G}_{x}}(\Omega)}+[u]_{s(.,.), G_{x, y}},
\end{equation}
where $[.]_{s(.,.), G_{x, y}}$ is the so called Gagliardo seminorm defined by
$$
[u]_{s(.,.), G_{x, y}}=\inf \left\{\lambda>0: J_{s(.,.),G_{x, y}}\left( \frac{u}{\lambda}\right)  \leqslant 1\right\} .
$$

\begin{remark}[{\textnormal{see} \cite{srtM2024eignv}}] Since the assumption $(g_{4})$ implies that the function $G_{x, y}$ and $\widetilde{G}_{x, y}$ satisfy the $\Delta_{2}$-condition, then the space $W^{s(.,.),G_{x, y}}(\Omega)$ is a reflexive and separable Banach space. Moreover, if $t\rightarrow G_{x, y}(\sqrt{t})$ is convex on $[0,+\infty)$, then the space $W^{s(.,.),G_{x, y}}(\Omega)$ is a uniformly convex space.
\end{remark}
Now, let us recall the following technical and important results.
\begin{lemma}[{\cite[Lemma 2.2]{abss2022class}}]\label{lemmapro}
Suppose that the assumptions $(\hyperref[cond:G1]{g_{1}})-(\hyperref[cond:G1]{g_{4}})$ hold. Then, the function $G_{x, y}$ satisfies the following inequalities  $:$
\begin{equation}
G_{x, y}(\delta t) \geqslant \delta^{g^{-}} G_{x, y}(t) \quad \quad \forall t>0,\quad \forall\delta>1,
\end{equation}
\begin{equation}
G_{x, y}(\delta t) \geqslant \delta^{g^{+}} G_{x, y}(t),\quad \forall t>0, \quad \forall\delta \in(0,1),
\end{equation}
\begin{equation}
G_{x, y}(\delta t) \leqslant \delta^{g^{+}} G_{x, y}(t) \quad \forall t>0,\quad \forall\delta>1,
\end{equation}
\begin{equation}
G_{x, y}(\delta t) \leqslant \delta^{g^{-}} G_{x, y}\left(t\right)\quad \forall t>0,\quad \forall \delta \in(0,1).
\end{equation}
\end{lemma}

\begin{theorem}\label{theo4}
Let $\Omega$ be an open subset of $\mathbb{R}^N$. Assume that $G_{x,y}\in \mathcal{B}_{f}$. Then, $C_0^{2}(\Omega)\subset W^{s(.,.),G_{x,y}}(\Omega)$.
\end{theorem}
 The proof of the previous theorem is similar to the one of Theorem 2.2 in \cite{abss2022class}.

\begin{theorem}[{\cite[Theorem 3.1]{biswas2021variable}}]\label{theo3} Let $\Omega$ be a smooth bounded domain in $\mathbb{R}^N$ and $s(.,.)$, $p(.,.)$ satisfy \eqref{S1}-\eqref{S2} and \eqref{P1}-\eqref{P2}, respectively. Let $r: \overline{\Omega} \longrightarrow(1,+\infty)$ be a continuous variable exponent such that

$$
1<r^{-}=\min _{x \in \overline{\Omega}} r(x) \leqslant r(x)<p_s^*(x):=\frac{N \overline{p}(x)}{N-\overline{s}(x) \overline{p}(x)} \quad \text { for all } x \in \overline{\Omega},
$$
where $\overline{p}(x)=p(x,x)$ and $\overline{s}(x)=s(x,x)$. Then, there exists a constant $C=C(N, s, p, r, \Omega)>0$ such that for any $u \in W^{s(.,.), p(x, y)}(\Omega)$,

$$
\|u\|_{L^{r(x)}(\Omega)} \leqslant C\|u\|_{W^{s(.,.), p(x, y)}(\Omega)}.
$$
Thus, the space $W^{s(.,.), p(x, y)}(\Omega)$ is continuously embedded in $L^{r(x)}(\Omega)$ with $r(x) \in\left(1, p_s^*(x)\right)$, for all $x \in \overline{\Omega}$. Moreover, this embedding is compact.
\end{theorem}

\begin{proposition}[\cite{barv2018onanew}]\label{Leb1} Let $u \in L^{q(x)}(\Omega)$, then
\begin{itemize}

\item[1)] $\|u\|_{L^{q(x)}(\Omega)}<1($ resp. $=1,>1) \Leftrightarrow \rho_{q(.)}(u)<1($ resp. $=1,>1)$,

\item[2)] $\|u\|_{L^{q(x)}(\Omega)}<1 \Rightarrow\|u\|_{L^{q(x)}(\Omega)}^{q+} \leqslant \rho_{q(.)}(u) \leqslant\|u\|_{L^{q(x)}(\Omega)}^{q-},$
\item[3)] $\|u\|_{L^{q(x)}(\Omega)}>1 \Rightarrow\|u\|_{L^{q(x)}(\Omega)}^{q-} \leqslant \rho_{q(.)}(u) \leqslant\|u\|_{L^{q(x)}(\Omega)}^{q+}$,
\end{itemize}
where $\displaystyle \rho_{q(.)}(u)=\int_{\Omega}|u(x)|^{q(x)}dx.$
\end{proposition}
\section{Proofs of main results}\label{sec3}
 The aim of this section is to prove the existence of a weak solution for the nonlocal singular elliptic problem \eqref{Prob1}. For this, we need to assume that $G_{x,y}$ satisfies the following condition 
\begin{equation}\label{ConditionN}
|t|^{p(x,y)}\leq M G_{x,y}(t)\quad \text{ for all } (x,y)\in \overline{\Omega}\times \overline{\Omega}, \text{ and all } t\geq 0,
\end{equation}
 where $M$ is a positive constant, and $p \in C(\overline{\Omega} \times \overline{\Omega})$ is as given in \eqref{P1}-\eqref{P2}.\\
  We denote by $\langle .,. \rangle$  the duality pairing between $W^{s(.,.),G_{x,y}}(\Omega)$ and its dual space $(W^{s(.,.),G_{x,y}}(\Omega))^{*}$.
\begin{remark} By \eqref{ConditionN} we deduce that $W^{s(.,.),G_{x,y}}(\Omega)$ is continuously embedded in $W^{s(.,.), p(x, y)}(\Omega)$. On the other hand, by Theorem \ref{theo3}, $W^{s(.,.), p(x, y)}(\Omega)$ is continuously embedded in $L^{r(x)}(\Omega)$ for any $r \in C(\overline{\Omega})$ with $1<r^{-} \leqslant r(x)<p_s^*(x)$ for all $x \in \overline{\Omega}$. Thus, the space $W^{s(.,.),G_{x,y}}(\Omega)$ is continuously embedded in $L^{r(x)}(\Omega)$ for any $r \in C(\overline{\Omega})$ with $1<r^{-} \leqslant r(x)<p_s^*(x)$ for all $x \in \overline{\Omega}$, that is, there exists a positive constant $C_0>0$ such that for any $u \in W^{s(.,.),G_{x,y}}(\Omega)$
$$
\|u\|_{L^{r(x)}(\Omega)} \leqslant C_0\|u\|_{W^{s(.,.),G_{x,y}}(\Omega)}.
$$
Moreover, this embedding is compact.
\end{remark}

In order to prove the existence result, we consider the energy functional $J: W^{s(.,.),G_{x,y}}(\Omega) \longrightarrow \mathbb{R}$ associated with problem \eqref{Prob1}, wich is defined by:

$$
J(u)  =\int_{\Omega}\int_{\Omega} G_{x,y}\left(D_{s(.,.)}u\right) d \mu+\int_{\Omega} \widehat{G}_{x}(|u|) d x-\int_{\Omega} \frac{g(x)|u|^{1-m(x)}}{1-m(x)} d x
$$
If we set
$$
\Psi(u)=\int_{\Omega} \int_{\Omega} G_{x,y}\left(D_{s(.,.)}u\right) \frac{\mathrm{d} x \mathrm{~d} y}{|x-y|^N}+\int_{\Omega} \widehat{G}_x(|u(x)|) \mathrm{d} x$$  and $$I(u)=\int_{\Omega} \frac{g(x)|u|^{1-m(x)}}{1-m(x)} d x.
$$
Then
$$
J(u)=\Psi(u)-I(u).
$$
\begin{proposition}[{\cite{srtM2024eignv}}]\label{prop1}
 Assume that assumption $(\hyperref[cond:G1]{g_{4}})$ is satisfied. Then for any $u\in W^{s(.,.), G_{x, y}}\left(\Omega\right) $, we have:
 $$ \|u\|_{W^{s(.,.), G_{x, y}}(\Omega)}>1 \Rightarrow \|u\|_{W^{s(.,.), G_{x, y}}(\Omega)}^{g^{-}} \leq \Psi(u) \leq  \|u\|_{W^{s(.,.), G_{x, y}}(\Omega)}^{g^{+}},$$
 $$ \|u\|_{W^{s(.,.), G_{x, y}}(\Omega)}<1 \Rightarrow \|u\|_{W^{s(.,.), G_{x, y}}(\Omega)}^{g^{+}} \leq \Psi(u) \leq  \|u\|_{W^{s(.,.), G_{x, y}}(\Omega)}^{g^{-}}.$$
\end{proposition}
 Proceeding as in \cite[Lemma 3.1 and Lemma 3.3]{abss2022class}, we obtain the following lemma:
\begin{lemma}\label{lemmaN}
Suppose that the condition $(\hyperref[cond:G1]{g_{4}})$ hold true. Then, we have
\begin{itemize}
\item[i)]  $\Psi\in C^{1}( W^{s(.,.),G_{x,y}}(\Omega),\mathbb{R})$ and 
$$\langle\Psi^{'}(u),v\rangle=\int_{\Omega}\int_{\Omega} a_{x,y}\left(\left|D_{s(.,.)}u\right|\right) D_{s(.,.)}u D_{s(.,.)}v d \mu+\int_{\Omega} \widehat{a}_{x}(|u|) u v d x $$
for all $u,v\in W^{s(.,.),G_{x,y}}(\Omega)$.
\item[ii)] $\Psi$  is sequentially weakly lower semicontinuous.
\end{itemize}
\end{lemma}

Now, we give an auxiliary result.
\begin{lemma}
Assume that the hypotheses of Theorem \ref{Theopr} are fulfilled. Then the functional $J$ is coercive.
\end{lemma}
\begin{proof}
 Let $u \in W^{s(.,.),G_{x,y}}(\Omega)$ with $\|u\|_{W^{s(.,.),G_{x,y}}(\Omega)}>1$. From Proposition \ref{prop1}, we have
$$
\begin{aligned}
J(u) & =\int_{\Omega}\int_{\Omega} G_{x,y}\left(D_{s(.,.)}u\right) d \mu+\int_{\Omega} \widehat{G}_{x}(|u|) d x-\int_{\Omega} \frac{h(x)|u|^{1-m(x)}}{1-m(x)} d x \\
& \geq\|u\|_{W^{s(.,.),G_{x,y}}(\Omega) }^{g^{-}}-\frac{\|h\|_{\infty}}{1-m^{+}}\|u\|_{L^{\frac{r(x)}{1-m(x)}}(\Omega)}^{1-m(x)}|1|_{L^{\frac{r(x)}{r(x)+m(x)-1}}(\Omega)} \\
& \geq\|u\|_{W^{s(.,.),G_{x,y}}(\Omega)}^{g^{-}}-c\|u\|_{W^{s(.,.),G_{x,y}}(\Omega)}
\end{aligned}
$$
since $g^{-}>1-m^{-}$the above inequality implies that $J(u) \to \infty$ as $\|u\|_{W^{s(.,.),G_{x,y}}(\Omega)} \to \infty$, that is, $J$ is coercive.
\end{proof}
\begin{lemma}\label{Lemm3}
 Assume that the hypotheses of Theorem \ref{Theopr} are fulfilled. Then there exists $v_{0} \in W^{s(.,.),G_{x,y}}(\Omega)$ such that $v_{0}>0$ and $J(t v_{0})<0$ for $t>0$ small enough.
\end{lemma}
\begin{proof}
Let $v_{0} \in C_0^{\infty}\left(\Omega\right)\setminus\{0\}$. By Theorem \ref{theo4} we have $\|v_{0}\|_{W^{s(.,.),G_{x,y}}(\Omega)}<\infty$. Hence, for $0<t<1$, we have
$$
\begin{aligned}
J(t v_{0}) & =\int_{\Omega}\int_{\Omega} G_{x,y}\left(\frac{t v_{0}(x)-t v_{0}(y)}{|x-y|^{s(x,y)}}\right) d \mu \\
& +\int_{\Omega} \widehat{G}_{x}(|t v_{0}|) d x-\int_{\Omega} \frac{h(x)|t v_{0}|^{1-m(x)}}{1-m(x)} d x \\
& \leq\|t v_{0}\|_{W^{s(.,.),G_{x,y}}(\Omega)}^{g^{-}}-\frac{t^{1-m^{+}}}{1-m^{+}} \int_{\Omega} h(x)|v_{0}|^{1-m(x)} d x \\
& \leq t^{g^{-}}\|v_{0}\|_{W^{s(.,.),G_{x,y}}(\Omega)}^{g^{-}}-\frac{t^{1-m^{+}}}{1-m^{+}} \int_{\Omega} h(x)|v_{0}|^{1-m(x)} d x
\end{aligned}
$$
since $g^{-}>1-m^{+}$and $\displaystyle\int_{\Omega} h(x)|v_{0}|^{1-m(x)} d x>0$ we have $J\left(t_0 v_{0}\right)<0$ for $t_0 \in(0, t)$ sufficiently small.
\end{proof}
\begin{lemma}
The functional $J$ attains its global minimum in $W^{s(.,.),G_{x,y}}(\Omega)$, that is, there exists a function $u_0 \in$ $W^{s(.,.),G_{x,y}}(\Omega)$ such that
$$
l=J\left(u_0\right)=\inf _{u \in W^{s(.,.),G_{x,y}}(\Omega)} J(u).
$$
\end{lemma}

\begin{proof}
 Let $\left(u_k\right)_k \subset W^{s(.,.), G_{x,y}}(\Omega)$ be a minimizing sequence for $J$, that is, a sequence satisfying
$$
\lim _{k \rightarrow \infty} J\left(u_k\right)=l .
$$
Since $J$ is coercive, $\left(u_k\right)_k \subset W^{s(.,.), G_{x,y}}(\Omega)$ is bounded in $W^{s(.,.), G_{x,y}}(\Omega)$. As $W^{s(.,.), G_{x,y}}(\Omega)$ is reflexive, therefore, up to a subsequence, there exists $u_0 \in W^{s(.,.), G_{x,y}}(\Omega)$ such that

$$
\begin{aligned}
&u_k \rightharpoonup u_0 \quad \text { weakly in } \quad W^{s(.,.), G_{x,y}}(\Omega), \\
&u_k \rightarrow u_0 \quad \text { strongly in } \quad L^{\widehat{G}_{x}}(\Omega), \\
&u_k(x) \rightarrow u_0(x) \quad \text { for a.e. } x \in \Omega .
\end{aligned}
$$
By Lemma \ref{lemmaN}, we have
\begin{equation}\label{Ineq1}
\Psi(u_{0})\leq \liminf _{k \rightarrow \infty} \Psi(u_{k}).
\end{equation}
We claim that
\begin{equation}\label{Ineq2}
\int_{\Omega} \frac{h(x)|u_{0}|^{1-m(x)}}{1-m(x)} d x=\lim _{k \rightarrow \infty} \int_{\Omega} \frac{h(x)|u_{k}|^{1-m(x)}}{1-m(x)} d x .
\end{equation}
It is clear that $$\lim _{k \rightarrow \infty} \frac{h(x)|u_{k}|^{1-m(x)}}{1-m(x)}=\frac{h(x)|u_{0}|^{1-m(x)}}{1-m(x)}\quad \text{for a.e.} x\in \Omega.$$
According to Vitali's theorem it suffices to show that the family $$\displaystyle \left\{ \frac{h(x)|u_{k}|^{1-m(x)}}{1-m(x)}, k\in \mathbb{N}\right\} $$
is uniformly absolutely continuous, which means:\\
Given $\epsilon >0$ there exist $\eta >0$, such that, if $|\Omega^{'}|<\eta$, then $$\displaystyle\int_{\Omega^{'}} \frac{h(x)|u_{k}|^{1-m(x)}}{1-m(x)} d x<\epsilon,\textnormal{ for all } k.$$
Let $\epsilon>0$, then, from Proposition \ref{Leb1} and using the absolute continuity of\\ $ \displaystyle\int_{\Omega}|g(x)|^{\frac{r(x)}{r(x)+m(x)-1}} d x $, there existe $\alpha,\eta>0$, such that, for every $\Omega^{'}\subset \Omega$ with $|\Omega^{'}|<\eta$, we have $$\|g\|^{\alpha}_{L^{\frac{r(x)}{r(x)+m(x)-1}}(\Omega^{'})}\leq\displaystyle\int_{\Omega^{'}}|g(x)|^{\frac{r(x)}{r(x)+m(x)-1}} d x\leq \epsilon^{\alpha}. $$
Thus, by Proposition \ref{Leb1} and Hölder’s inequality, we have
$$
\begin{aligned}
&\int_{\Omega^{'}} \frac{g(x)}{1-m(x)}|u_{k}|^{1-m(x)} d x\\&\leq\frac{1}{1-m^{+}}\|g\|_{L^{\frac{r(x)}{r(x)+m(x)-1}}(\Omega^{'})}\| | u_{k}|^{1-m(x)}\|_{L^{\frac{r(x)}{1-m(x)}}(\Omega^{'})}\\& \leq \frac{1}{1-m^{+}}\|g\|_{L^{\frac{r(x)}{r(x)+m(x)-1}}(\Omega^{'})}\max{ \left( \|u_{k}\|_{L^{r(x)}(\Omega^{'})}^{1-m^{-}},\|u_{k}\|_{L^{r(x)}(\Omega^{'})}^{1-m^{+}}\right) }.
\end{aligned}
$$
Thus,
$$
\int_{\Omega^{'}} \frac{g(x)}{1-m(x)}|u_{k}|^{1-m(x)} d x\leq\frac{\epsilon}{1-m^{+}}\max{\left(\|u_{k}\|_{L^{r(x)}(\Omega)}^{1-m^{-}},\|u_{k}\|_{L^{r(x)}(\Omega)}^{1-m^{+}}\right)}.
$$
Since $r(x)<p^*(x)$, then $\|u_{k}\|_{L^{r(x)}(\Omega)}$ is bounded, and this facts implies that \eqref{Ineq2} is valid. 
Hence, by \eqref{Ineq1} and \eqref{Ineq2}, we deduce that $$l\leq J(u_{0})\leq \liminf_{k\to \infty }J(u_{k})=l.$$
Consequently $J(u_{0})=l$.
\end{proof}
\begin{remark}
\begin{itemize}

\item[1)] From Lemma \ref{Lemm3} and the fact that $u_0$ is a global minimum, we get $l=J(u_{0})< 0=J(0)$,
thus $u_{0}\neq 0$.
\item[2)] It is clear that $$|D_{s(.,.)}|u_{0}||\leq |D_{s(.,.)}u_{0}|,$$
then $|u_{0}|\in W^{s(.,.), G_{x,y}}(\Omega)$, therefore $$J(|u_{0}|)\leq J(u_{0})= \inf_{u\in W^{s(.,.), G_{x,y}}(\Omega)}J(u),$$
this implies that $J(|u_{0}|)= J(u_{0})$. Hence, we can suppose that $u_0\geq 0$.
\end{itemize}
\end{remark}

\begin{proof}[Proof of Theorem~{\upshape\ref{Theopr}}]
Let $\varphi\geq 0$ and $t>0$, we have\\
$$\begin{aligned} & 0 \leq \liminf _{t \rightarrow 0} \frac{J(u_{0}+t \varphi)-J(u_{0})}{t} \leq \int_{\Omega}\int_{\Omega} a_{x,y}\left(\left|D_{s(.,.)}u_0\right|\right) D_{s(.,.)}u_0 D_{s(.,.)}\varphi d \mu \\ & +\int_{\Omega} \widehat{a}_{x}(|u_{0}|) u_{0} \varphi d x-\limsup _{t \rightarrow 0} \int_{\Omega} h(x) \frac{(u_{0}+t \varphi)^{1-m(x)}-u_{0}^{1-m(x)}}{t(1-m(x))} d x, \end{aligned}$$
 thus
\begin{equation}\label{Ineq3} 
 \begin{aligned} &\limsup _{t \rightarrow 0} \int_{\Omega} h(x) \frac{(u_{0}+t \varphi)^{1-m(x)}-u_{0}^{1-m(x)}}{t(1-m(x))} d x\\& \leq \int_{\Omega}\int_{\Omega} a_{x,y}\left(\left|D_{s(.,.)}u_0\right|\right) D_{s(.,.)}u_0 D_{s(.,.)}\varphi d \mu+\int_{\Omega} \widehat{a}_{x}(|u_{0}|) u_{0} \varphi d x\end{aligned}
 \end{equation}
By the mean value theorem, there exists $\epsilon \in (0,1)$ such that
$$
\int_{\Omega} h(x)\frac{(u_{0}+t\varphi)^{1-m(x)}-u_{0}^{1-m(x)}}{t(1-m(x))}dx
= \int_{\Omega} h(x)(u_{0}+t\epsilon\varphi)^{-m(x)}\varphi dx.
$$
Since $\varphi \geq 0$, by Fatou's Lemma, we get
\begin{equation}\label{Ineq4}
\begin{aligned}
&\limsup_{t\to 0} \int_{\Omega} h(x)\frac{(u_{0}+t\varphi)^{1-m(x)}-u_{0}^{1-m(x)}}{t(1-m(x))}dx
\\&\geq \liminf_{t\to 0} \int_{\Omega} h(x)\frac{(u_{0}+t\varphi)^{1-m(x)}-u_{0}^{1-m(x)}}{t(1-m(x))}dx
\\&= \liminf_{t\to 0} \int_{\Omega} h(x)(u_{0}+t\epsilon\varphi)^{-m(x)}\varphi dx
\geq \int_{\Omega} h(x)u_{0}^{-m(x)}\varphi dx \geq 0.
\end{aligned}
\end{equation}
Therefore, by \eqref{Ineq3} and \eqref{Ineq4}, we have
\begin{equation}\label{Eqt2}
\begin{aligned}
&\int_{\Omega}\int_{\Omega}a_{x,y}\left(\left|D_{s(.,.)}u_0\right|\right) D_{s(.,.)}u_0 D_{s(.,.)}\varphi\, d\mu\\
+ &\int_{\Omega} \widehat{a}_{x}(|u_{0}|)u_{0}\varphi dx
- \int_{\Omega} h(x)u_{0}^{-m(x)}\varphi dx \geq 0,
\end{aligned}
\end{equation}
for all $\varphi \in W^{s(.,.),G_{x,y}}(\Omega)$, with $\varphi \geq 0$.\\
As established above, we have $$(-\Delta)^{s(.,.)}_{G_{x,y}} u_{0}+a_{x,y}(|u_{0}|)u_{0}\geq 0 \textnormal{ weakly in } \Omega.$$ \\
Since $u_0\geq 0$ and $u_0\neq 0$, the strong maximum principle for weak solutions implies that $u_0(x)> 0$ for every $x \in \Omega$.\\
Next, let us prove that $u_{0}\in W^{s(.,.),G_{x,y}}(\Omega)$ satisfies \eqref{Eqt1}. Let us define $\gamma:[-\epsilon,\epsilon]\longrightarrow \mathbb{R}$ as $$\gamma(t)=J((1+t)u_{0}).$$ Thus $\gamma(0)=\inf_{t\in[-\epsilon,\epsilon]}\gamma(t)$, from where $\gamma^{'}(0)=0$.  It is straightforward to see that $u_{0}$ satisfies
\begin{equation}\label{Eqt3}
\begin{aligned}
&\int_{\Omega}\int_{\Omega} a_{x,y}\left(\left|D_{s(.,.)}u_0\right|\right) D_{s(.,.)}u_0 D_{s(.,.)}u_0\, d\mu\\
+ &\int_{\Omega} \widehat{a}_{x}(|u_{0}|)u_{0}^{2}dx
- \int_{\Omega} h(x)u_{0}^{1-m(x)} dx = 0.
\end{aligned}
\end{equation}
Let us consider the following sets $$ \Omega_{\epsilon}=\left\lbrace x\in \Omega : u_{0}+\epsilon \varphi < 0\right\rbrace  \text{ and } \Omega^{\epsilon}=\left\lbrace x\in \Omega : u_{0}+\epsilon \varphi \geqslant 0\right\rbrace,
$$
 and define the following functions $$\begin{aligned} &\varphi_{\epsilon}=u_{0}+\epsilon \varphi,\\
 &\varphi_{\epsilon}^{+}=\max\left\lbrace u_{0}+\epsilon \varphi,0\right\rbrace ,\\ 
  &\varphi_{\epsilon}^{-}=\max\left\lbrace -(u_{0}+\epsilon \varphi),0\right\rbrace.  
 \end{aligned}$$
 It is clear that $$|D_{s(.,.)}\varphi_{\epsilon}(x,y)|\geqslant |D_{s(.,.)}\varphi_{\epsilon}^{+}(x,y)| \text{ a.e. } (x,y)\in \Omega\times \Omega.$$
 Therefore $\varphi_{\epsilon}^{-},\varphi_{\epsilon}^{+}\in W^{s(.,.), G_{x,y}}(\Omega)$. By choosing $\varphi_{\epsilon}^{+}$ in \eqref{Eqt2} we have
$$
\begin{aligned}
0 \leq & \int_{\Omega^{\epsilon}} \int_{\Omega^{\epsilon}}
a_{x,y}\!\left(\left|D_{s(.,.)}u_{0}\right|\right) D_{s(.,.)}u_{0} D_{s(.,.)}(u_{0}+\varepsilon\varphi) d\mu \\
& + \int_{\Omega^{\epsilon}} \widehat{a}_{x}(|u_{0}|)u_{0}(u_{0}+\epsilon\varphi)dx
 - \int_{\Omega^{\epsilon}} h(x)u_{0}^{-m(x)}(u_{0}+\epsilon\varphi)dx \\
= & \left( \int_{\Omega}\int_{\Omega} - \int_{\Omega_{\epsilon}} \int_{\Omega_{\epsilon}} \right) 
a_{x,y}\!\left(\left|D_{s(.,.)}u_{0}\right|\right)D_{s(.,.)}u_{0} D_{s(.,.)}(u_{0}+\epsilon\varphi)d\mu \\
 + &\left( \int_{\Omega} - \int_{\Omega^{\epsilon}} \right) \widehat{a}_{x}(|u_{0}|)u_{0}(u_{0}+\epsilon\varphi)dx
 - \left( \int_{\Omega} - \int_{\Omega^{\epsilon}} \right) h(x)u_{0}^{-m(x)}(u_{0}+\epsilon\varphi)dx \\
= & \int_{\Omega}\int_{\Omega} a_{x,y}\!\left(\left|D_{s(.,.)}u_{0}\right|\right) D_{s(.,.)}u_{0} D_{s(.,.)}u_{0} d\mu 
+ \int_{\Omega} \widehat{a}_{x}(|u_{0}|)u_{0}^2 dx - \int_{\Omega} h(x)u_{0}^{1-m(x)}dx \\
 +& \epsilon \int_{\Omega}\int_{\Omega} a_{x,y}\!\left(\left|D_{s(.,.)}u_{0}\right|\right) D_{s(.,.)}u_{0} D_{s(.,.)}\varphi d\mu 
+ \epsilon \int_{\Omega} \widehat{a}_{x}(|u_{0}|)u_{0}\varphi dx \\
 - &\epsilon \int_{\Omega} h(x)u_{0}^{-m(x)}\varphi dx - \int_{\Omega_{\epsilon}} \int_{\Omega_{\epsilon}}
a_{x,y}\!\left(\left|D_{s(.,.)}u_{0}\right|\right) D_{s(.,.)}u_{0} D_{s(.,.)}(u_{0}+\epsilon\varphi)d\mu \\
 - &\int_{\Omega_{\epsilon}} \widehat{a}_{x}(|u_{0}|)u_{0}(u_{0}+\epsilon\varphi)dx
+ \int_{\Omega_{\epsilon}} h(x)u_{0}^{-m(x)}(u_{0}+\epsilon\varphi)dx 
\end{aligned}
$$
\begin{equation}\label{Innn1}
\begin{aligned}
= & \ \epsilon \left( \int_{\Omega}\int_{\Omega} a_{x,y}\!\left(\left|D_{s(.,.)}u_{0}\right|\right)D_{s(.,.)}u_{0} D_{s(.,.)}\varphi d\mu \right) \\
& + \epsilon \left( \int_{\Omega}\widehat{a}_{x}(|u_{0}|)u_{0}\varphi dx - \int_{\Omega} h(x)u_{0}^{-m(x)}\varphi dx \right) \\
& - \int_{\Omega_{\epsilon}} \int_{\Omega_{\varepsilon}}
a_{x,y}\!\left(\left|D_{s(.,.)}u_{0}\right|\right) D_{s(.,.)}u_{0} D_{s(.,.)}(u_{0}+\epsilon\varphi) d\mu \\
& - \int_{\Omega_{\epsilon}} \widehat{a}_{x}(|u_{0}|)u_{0}(u_{0}+\epsilon\varphi)dx
+ \int_{\Omega_{\epsilon}} h(x)u_{0}^{-m(x)}(u_{0}+\epsilon\varphi)dx \\
\leq & \ I_{1} - I_{2}
\end{aligned}
\end{equation}
where
$$
\begin{aligned}
 \ I_{1} &=
\epsilon \left( \int_{\Omega}\int_{\Omega} a_{x,y}\!\left(\left|D_{s(.,.)}u_{0}\right|\right)
D_{s(.,.)}u_{0} D_{s(.,.)}\varphi d\mu \right)\\
&+ \epsilon \left( \int_{\Omega} \widehat{a}_{x}(|u_{0}|)u_{0}\varphi dx
- \int_{\Omega} h(x)u_{0}^{-m(x)}\varphi dx \right)
\end{aligned}
$$
 and
$$
\begin{aligned}
 \ I_{2} &=
\epsilon \int_{\Omega_{\varepsilon}} \int_{\Omega_{\varepsilon}}
a_{x,y}\!\left(\left|D_{s(.,.)}u_{0}\right|\right) D_{s(.,.)}u_{0} D_{s(.,.)}\varphi d\mu\\
&- \epsilon \int_{\Omega_{\varepsilon}} \widehat{a}_{x}(|u_{0}|)u_{0}\varphi dx .
\end{aligned}
$$
Considering that $u_{0} > 0$ and the Lebesgue measure of the integration domain 
$\Omega_{\varepsilon}$ tends to zero as $\epsilon \to 0$, we have
\[
\int_{\Omega_{\varepsilon}} \int_{\Omega_{\epsilon}}
a_{x,y}\!\left(\left|D_{s(.,.)}u_{0}\right|\right) D_{s(.,.)}u_{0} D_{s(.,.)}\varphi d\mu 
\to 0, \quad \text{as } \epsilon \to 0,
\]
and
\[
\int_{\Omega_{\epsilon}} \widehat{a}_{x}(|u_{0}|)u_{0}\varphi dx \to 0, 
\quad \text{as } \varepsilon \to 0.
\]
Moreover, considering that $a_{x,y}(\cdot) \in (0,\infty)$, we can drop the term
\[
-\int_{\Omega_{\epsilon}} \int_{\Omega_{\epsilon}}
a_{x,y}\!\left(\left|D_{s(.,.)}u_{0}\right|\right) D_{s(.,.)}u_{0} D_{s(.,.)}u_{0} d\mu
- \int_{\Omega^{\epsilon}} \widehat{a}_{x}(|u_{0}|)u_{0}^2 dx,
\]
in \eqref{Innn1} as it is negative. Therefore, by dividing \eqref{Innn1} 
by $\epsilon$ and letting $\epsilon \to 0$, we find
\begin{equation}\label{Eqt4}
\int_{\Omega}\int_{\Omega} a_{x,y}\!\left(\left|D_{s(.,.)}u_{0}\right|\right)D_{s(.,.)}u_{0} D_{s(.,.)}\varphi d\mu
+ \int_{\Omega} \widehat{a}_{x}(|u_{0}|)u_{0}\varphi dx
- \int_{\Omega} h(x)u_{0}^{-m(x)}\varphi dx \geq 0. 
\end{equation}
Since $\varphi$ is arbitrary in $W^{s(.,.),G_{x,y}}(\Omega)$, then, we can replace $\varphi$ by $-\varphi$ in \eqref{Eqt4} , which implies that 
\begin{equation}\label{Eqt5}
\int_{\Omega}\int_{\Omega} a_{x,y}\!\left(\left|D_{s(.,.)}u_{0}\right|\right) D_{s(.,.)}u_{0} D_{s(.,.)}\varphi d\mu
+ \int_{\Omega} \widehat{a}_{x}(|u_{0}|)u_{0}\varphi dx
- \int_{\Omega} h(x)u_{0}^{-m(x)}\varphi dx \leqslant 0. 
\end{equation}
Combining \eqref{Eqt4} and \eqref{Eqt5}, we obtain 
\begin{equation}\label{Eqt6}
\int_{\Omega}\int_{\Omega} a_{x,y}\!\left(\left|D_{s(.,.)}u_{0}\right|\right) D_{s(.,.)}u_{0} D_{s(.,.)}\varphi d\mu
+ \int_{\Omega}\widehat{a}_{x}(|u_{0}|)u_{0}\varphi dx
- \int_{\Omega} h(x)u_{0}^{-m(x)}\varphi dx = 0.
\end{equation}
for all $\varphi\in W^{s(.,.),G_{x,y}}(\Omega) $. This concludes the proof.

\end{proof}

\section{Examples}
In this section, we present some examples of functions $G_{x,y}$ for which the existence result Theorem \ref{Theopr} may be applied.
\begin{itemize}
\item[1)] As a first example,we can take $G_{x,y}(t)=|t|^{p(x,y)}$ where $p\in C(\overline{\Omega} \times \overline{\Omega})$ satisfies $2\leqslant p(x,y)<N$ for all $(x,y)\in \overline{\Omega} \times \overline{\Omega}$.\\
  In this case, the problem \eqref{Prob1} reduces to the following variable order fractional $p(x,.)$-Laplacian problem
   \begin{equation}\label{Prob2}
\begin{cases}
(-\Delta)^{s(.,.)}_{p(x,.)} u + |u|^{\overline{p}(x)}u =  h(x)u^{-m(x)}, & \text{ in } \Omega, \\
u = 0, & \text{ on } \mathbb{R}^{N}\setminus \Omega, 
\end{cases}
\end{equation}
where $\overline{p}(x)=p(x,x)$ for all $x\in \overline{\Omega}$. Here, the operator $(-\Delta)^{s(.,.)}_{p(x,\cdot)}$ is the variable order fractional $p(x,\cdot)$-Laplacian
operator defined as follows
\[
(-\Delta)^{s(.,.)}_{p(x,\cdot)}u(x)
= 2 \lim_{\epsilon \to 0^+}\int_{\Omega}
\frac{|u(x)-u(y)|^{p(x,y)-2}\,(u(x)-u(y))}{|x-y|^{N+s(x,y) p(x,y)}}\, dy
\qquad \text{for all } x\in \Omega.
\]
It is clear that the generalized $N$-function $G_{x,y}$ satisfies the assumptions $(\hyperref[cond:G1]{g_{1}})-(\hyperref[cond:G1]{g_{4}})$, \eqref{ConditionN} and $G_{x,y}\in \mathcal{B}_{f}$. In this case, we can take $g^{-}=p^{-}$ and $g^{+}=p^{+}$. Then, we can extract the following result.

\medskip

\begin{remark}   
Then problem \eqref{Prob2} has a nontrivial weak solution  
$u_{0}\in W^{s(.,.),p(x,y)}(\Omega)$.
\end{remark}
\item[2)]  As a second example,we can take $$\displaystyle G_{x,y}(t)=p(x,y)\frac{|t|^{p(x,y)}}{log(1+|t|)}+\int_0^{|t|}\frac{\tau^{p(x,y)}}{(1+\tau)(log(1+\tau))^2}d\tau,$$
where $p\in C(\overline{\Omega} \times \overline{\Omega})$ satisfies $2\leqslant p(x,y)<N$ for all $(x,y)\in \overline{\Omega} \times \overline{\Omega}$.\\
 Then, in this case problem \eqref{Prob1} redeuces to the problem 
   \begin{equation}\label{Prob3}
\begin{cases}
\displaystyle(-\Delta)^{s(.,.)}_{G_{x,y}} u(x) + \frac{\overline{p}(x)|u|^{\overline{p}(x)-2}u}{log(1+|u|) }=  h(x)u^{-m(x)}, & \text{ in } \Omega, \\
u = 0, & \text{ on } \mathbb{R}^{N}\setminus \Omega. 
\end{cases}
\end{equation}
Here the operator $(-\Delta)^{s(.,.)}_{G_{x,y}}$ is defined by
$$(-\Delta)^{s(.,.)}_{G_{x,y}}u(x)=2 \lim_{\epsilon \to 0^+}\int_{\mathbb{R}^N \setminus B_\epsilon (x)}\frac{p(x,y)|D_{s(.,.)}u|^{p(x,y)-2}D_{s(.,.)}u}{log(1+|D_{s(.,.)}u|) |x-y|^{N+s(x,y)}}dy \quad \text{ for all } x\in \Omega.$$
It is easy to see that $G_{x,y}$ is a generalized $N$-function and satisfies the assumptions $(\hyperref[cond:G1]{g_{1}})-(\hyperref[cond:G1]{g_{4}})$, \eqref{ConditionN} and $G_{x,y}\in \mathcal{B}_{f}$. In this case, we can take $g^{-}=p^{-}$ and $g^{+}=p^{+}$. Then, we can extract the following result.

\begin{remark} 
Then problem \eqref{Prob3} has a nontrivial weak solution  
$u_{0}\in W^{s(.,.),p(x,y)}(\Omega)$.
\end{remark}
\item[3)] As a third example,we can take $$\displaystyle G_{x,y}(t)=|t|^{p(x,y)}log(1+|t|),$$
where $p\in C(\overline{\Omega} \times \overline{\Omega})$ satisfies $2\leqslant p(x,y)<N$ for all $(x,y)\in \overline{\Omega}\times \overline{\Omega}$.\\
Then, the problem \eqref{Prob1} becomes 
\begin{equation}\label{Prob4}
\begin{cases}
\displaystyle(-\Delta)^{s(.,.)}_{G_{x,y}} u(x) + \overline{p}(x)|u|^{\overline{p}(x)-2}ulog(1+|u|)+\frac{|u|^{\overline{p}(x)}}{1+|u|}=  h(x)u^{-m(x)}, & \text{ in } \Omega, \\
u = 0, & \text{ on } \mathbb{R}^{N}\setminus \Omega. 
\end{cases}
\end{equation}
with  $$
\begin{aligned}(-\Delta)^{s(.,.)}_{G_{x,y}}u(x)=2 \lim_{\epsilon \to 0^+}\int_{\mathbb{R}^N \setminus B_\epsilon (x)}&p(x,y)|D_{s(.,.)}u|^{p(x,y)-2}log(1+|D_{s(.,.)}u|)\\&+\frac{|D_{s(.,.)}u|^{p(x,y)-1}}{1+|D_{s(.,.)}u|}\frac{dy}{|x-y|^{N+s(x,y)}}
\end{aligned}$$
 \end{itemize}
for all $x\in \Omega$. It is easy to see that $G_{x,y}$ is a generalized $N$-function and satisfies the assumptions $(\hyperref[cond:G1]{g_{1}})-(\hyperref[cond:G1]{g_{4}})$, \eqref{ConditionN} and $G_{x,y}\in \mathcal{B}_{f}$. In this case, we can take $g^{-}=p^{-}$ and $g^{+}=p^{+}$. Then, we can extract the following result.

\begin{remark} 
Then problem \eqref{Prob4} has a nontrivial weak solution  
$u_{0}\in W^{s(.,.),p(x,y)}(\Omega)$.
\end{remark}
\section*{Compliance with ethical standards}
\textbf{Conflict of interest}  The authors declare that they have no conflict of interest.
\bibliographystyle{plain}

% ------------------------------------------------------------------------
\end{document}